\documentclass[12pt,a4paper]{amsart}
\usepackage{amscd}
\theoremstyle{plain} 

\begingroup

\newtheorem{theorem}{Theorem}[section]

\newtheorem{cor}[theorem]{Corollary}

\newtheorem{lemma}[theorem]{Lemma}

\newtheorem{prop}[theorem]{Proposition}

\endgroup

\theoremstyle{definition}

\newtheorem{definition}[theorem]{Definition}

\theoremstyle{remark}

\begin{document}

\title[On  subanalytic subsets of real analytic orbifolds]
{On  subanalytic subsets of real analytic orbifolds}

\author{Marja Kankaanrinta}
\address{Department of Mathematics \\
         P.O. Box 400137\\ University of Virginia\\
         Charlottesville VA 22904 - 4137\\ U.S.A.}
\email{mk5aq@@virginia.edu}

\date{\today}



\begin{abstract} The purpose of this paper is to define semi-
and subanalytic subsets and maps in the context of real analytic
orbifolds and to study their basic properties. We prove results
analogous to some well-known results  in the manifold case.
For example, we prove that if $A$ is a subanalytic subset of a real
analytic quotient orbifold $X$, then there is a real analytic orbifold $Y$ of
the same dimension as $A$ and a proper real analytic map
$f\colon Y\to X$ with $f(Y)=A$. We also study images and inverse images of
subanalytic sets and show that if $X$ and $Y$ are real analytic orbifolds and if
$f\colon X\to Y$ is a subanalytic map, then
the inverse image $f^{-1}(B)$ of any subanalytic subset $B$ of $Y$ is 
subanalytic. If, in addition, $f$ is proper, then also the image $f(A)$ of
any subanalytic subset $A$ of $X$ is proper.
\end{abstract}
\maketitle
\section{Introduction}
\label{intro}

\noindent Semianalytic subsets of real analytic manifolds are locally
defined by a finite number of equalities and inequalities of real analytic
maps. 
Thus the definition of a semianalytic set is analogous to that of a semialgebraic
set. Consequently, the theory of semianalytic sets resembles that of semialgebraic
sets. However, there are essential differences.  While the property of being
semialgebraic is preserved by a rational map, the image of a semianalytic
set  by a real analytic map is not necessarily semianalytic, not even if the
map is proper. 

From many points of view, for example when one considers such 
topics as triangulations or stratifications, the images of semianalytic sets by proper
real analytic maps are as good as semianalytic sets.
This fact led to the concept of subanalytic sets.
Subanalytic sets form the smallest class of sets that contains all the
semianalytic sets and is closed under the operation of taking images by
proper real analytic maps. Locally,
subanalytic sets are just projections of relatively compact
semianalytic sets. {\L}ojasiewicz was the first one to study properties of
semianalytic and subanalytic sets, see \cite{Lo}, although the
word "subanalytic" is due to Hironaka (\cite{Hi1}).
The theory of semianalytic and subanalytic
sets was then developed by Gabrielov (\cite{Ga}), Hardt (\cite{Ha1},
\cite{Ha}) and Hironaka (\cite{Hi1}, \cite{Hi2}).

Orbifolds are generalizations of manifolds.
The concept of an orbifold was originally introduced in 1957 by Satake (\cite{S}),  who
used the term  "$V$-manifold". The word "orbifold" is due to Thurston
(\cite{Th}), who used orbifolds to study the structure of $3$-manifolds
in the 1970's. Locally, an $n$-dimensional orbifold, 
where $n\in{\mathbb{N}}$,
is an orbit space of a finite group action on an open connected 
subset of ${\mathbb{R}}^n$. 
Maps between orbifolds that take into account the local orbit space
structure, i.e., maps that locally are induced by equivariant maps, are
called orbifold maps.  

The purpose of this paper is to generalize the concepts of semianalytic and
subanalytic sets to the orbifold case.  
We begin by discussing real analytic orbifolds in Sections \ref{real},
\ref{esimerkki} and \ref{suborbi}. 
We show that the quotient space of
a real analytic manifold by a proper real analytic almost free action of a
Lie group is a real analytic orbifold (Theorem \ref{quotorb}) and that
every suborbifold of a real analytic quotient orbifold also is a quotient orbifold
(Theorem \ref{apusub}).

Semi- and subanalytic subsets of real analytic orbifolds are introduced in
Section \ref{semisubdefi}.
We  show that the  closures,
interiors, complements, finite unions and finite intersections of semianalytic  
(subanalytic) subsets of real analytic orbifolds are semianalytic (subanalytic)
(Theorem \ref{basicprop}). 

We continue by defining semianalytic
and subanalytic maps between two orbifolds and show that
the definitions are analogous to those  
of a semianalytic and a subanalytic map between two real analytic
manifolds (Theorem \ref{graafi}).
Basic properties of semianalytic and subanalytic maps are considered
in the orbifold case. In particular, we show that the inverse image of a
subanalytic set by a subanalytic map is subanalytic and that the image
of a subanalytic set by a proper subanalytic map is subanalytic (Theorem
\ref{kuvat}). 

By definition, subanalytic subsets of real analytic manifolds are  locally
projections of relatively compact semianalytic sets. In Section 
\ref{altern}, we show that subanalytic subsets of real analytic orbifolds
can be characterized by the same property (Theorem \ref{orbiss}).

We also prove a version of the  uniformization theorem for
subanalytic subsets of
quotient orbifolds, i.e., we show that if $X$ is a real analytic quotient orbifold
and if $A$ is a closed subanalytic subset of $X$, then there exists a real analytic
orbifold $Y$ of the same dimension as $A$ and a proper real analytic
map $f\colon Y\to X$ such that $f(Y)=A$ (Theorem \ref{important}).

Results of this paper are applied in \cite{Ka3}, where subanalytic
triangulations of real analytic orbifolds are studied.

\section{Real analytic orbifolds}
\label{real}

\noindent In this section we recall the definition and some basic properties
of an orbifold.

\begin{definition}
\label{ensin}
Let $X$ be a topological space and let $n> 0$.
{\begin{enumerate}
\item An $n$-dimensional {\it orbifold chart} for an open subset $V$
of  $X$ is a triple
$(\tilde{V}, G,\varphi)$ such that
{\begin{enumerate}
\item $\tilde{V}$ is a connected open subset of ${\mathbb{R}}^n$,

\item $G$ is a finite group of homeomorphisms acting  on $\tilde{V}$, let
${\rm ker}(G)$ denote the subgroup of $G$ acting trivially on $\tilde{V}$.

\item $\varphi\colon \tilde{V}\to V$ is a $G$-invariant map inducing a
homeomorphism from the orbit space $\tilde{V}/G$ onto  $V$.
\end{enumerate}}

\item  If $V_i\subset V_j$, an {\it embedding} $(\lambda_{ij}, h_{ij})
\colon (\tilde{V}_i, G_i, \varphi_i)\to
(\tilde{V}_j, G_j,\varphi_j)$ between two orbifold charts is 
{\begin{enumerate}
\item an injective homomorphism $h_{ij}\colon G_i\to G_j$, 
such that $h_{ij}$ is an isomorphism from ${\rm ker}(G_i)$ to
${\rm ker}(G_j)$,
and
\item an equivariant embedding
$\lambda_{ij}\colon \tilde{V}_i\to \tilde{V}_j$ such that $\varphi_j\circ
\lambda_{ij}=\varphi_i$. (Here the equivariantness means that 
$\lambda_{ij}(gx)=h_{ij}(g)\lambda_{ij}(x)$ for every $g\in G_i$
and for every $x\in \tilde{V}_i$.)
\end{enumerate}}

\item  An {\it orbifold atlas} on $X$ is a family ${\mathcal V}=\{ (\tilde{V_i}, G_i, \varphi_i)\}_{
i\in I}$ of orbifold charts such that
{\begin{enumerate}
\item $\{ V_i\}_{i\in I}$ is a covering of $X$,

\item  given two charts $(\tilde{V}_i, G_i, \varphi_i)$ and
$(\tilde{V}_j, G_j, \varphi_j)$ and a point $x\in V_i\cap V_j$, there exists
an open neighborhood $V_k\subset V_i\cap V_j$ of $x$ and a chart
$(\tilde{V}_k, G_k, \varphi_k)$ such that there are embeddings 
$(\lambda_{ki}, h_{ki})\colon (\tilde{V}_k,G_k, \varphi_k)\to
(\tilde{V}_i, G_i,\varphi_i)$ and
$(\lambda_{kj}, h_{kj})\colon (\tilde{V}_k,G_k, \varphi_k)\to
(\tilde{V}_j, G_j,\varphi_j)$.
\end{enumerate}}

\item An atlas ${\mathcal V}$ is called a {\it refinement} of another atlas
${\mathcal U}$ if for every chart in ${\mathcal V}$ there exists an embedding
into some chart of ${\mathcal U}$. Two orbifold atlases are called
{\it equivalent} if they have a common refinement.
\end{enumerate}}
\end{definition}
 
 Every orbifold atlas ${\mathcal V}$ is contained in a unique {\it maximal}
 orbifold atlas, i.e., in a maximal family of orbifold charts satisfying 
 conditions $(1)$, $(2)$ and $(3)$.
 
 \begin{definition}
 An $n$-dimensional {\it orbifold} is a paracompact Hausdorff space $X$
 equipped with an equivalence class of $n$-dimensional orbifold atlases.
 \end{definition}
 
 Without further  mentioning, we assume that every orbifold has only
 countably many connected components.

 The sets $V\in {\mathcal V}$ are called {\it basic open sets} in $X$.
 
 An orbifold is  called {\it real analytic} (resp. {\it smooth}, i.e., {\it differentiable of 
 class ${\rm C}^\infty$}), 
 if each $G_i$ acts via 
 real analytic (resp. smooth) diffeomorphisms on 
 $\tilde{V}_i$ and if each embedding $\lambda_{ij}\colon {\tilde{V}}_i\to
 \tilde{V}_j$ is real analytic (resp. smooth).

An $n$-dimensional orbifold $X$  is called {\it locally smooth}, if 
 for each $x\in X$ there is an orbifold chart $(\tilde{U}, G,\varphi)$
 with  $x\in U=\varphi(\tilde{U})$ and such that
 the action of $G$ on $\tilde{U}\cong {\mathbb{R}}^n$  
 is orthogonal. 
 By the 
 slice theorem (see Proposition 2.2.2 in \cite{Pa2} for the smooth
 case and Theorem 2.5 in \cite{Ka} for the real analytic version),
 all real analytic and smooth orbifolds are locally smooth.

Let $Y_1$ and $Y_2$ be real analytic orbifolds
with orbifold atlases ${\mathcal V}=\{ (\tilde{V}_i, G_i, \varphi_i)\}_{i\in I}$
and ${\mathcal U}=\{(\tilde{U}_j, H_j, \psi_j)\}_{j\in J}$, respectively.
Let the groups $G_i\times H_j$ act diagonally on the sets
$\tilde{V}_i\times \tilde{U}_j$. Then the product $Y_1\times Y_2$ is a
real analytic orbifold with the orbifold atlas ${\mathcal V}\times
{\mathcal U}=\{ ( \tilde{V}_i\times\tilde{U}_j, G_i\times H_j,
\varphi_i\times\psi_j)\}_{(i,j)\in I\times J}$.

\section{Real analytic quotient orbifolds}
\label{esimerkki}

\noindent Recall that a map $f\colon M\to N$ is called {\it proper}, if the 
inverse image
$f^{-1}(K)$ of any compact subset $K$ of $N$ is compact.
Let $G$ be a Lie group and let $M$ be a real analytic manifold
on which $G$ acts real analytically. If the action of $G$ on $M$
is  {\it proper}, i.e.,  if the map
$$
G\times M\to M\times M, (g,x)\mapsto(gx,x),
$$
is proper, we call $M$ a {\it proper real analytic $G$-manifold}.
Proper smooth $G$-manifolds are defined accordingly.
The action is called {\it almost free},  if every
isotropy group is finite.  
The following theorem shows that orbit spaces  of such actions
are orbifolds.  At least the smooth case of the theorem is known,
see \cite{AR} p. 536. The proof is presented here, since we failed to
find one in literature.

\begin{theorem}
\label{quotorb}
Let $G$ be a Lie group and let $M$ be a a proper  real
analytic (resp. smooth) 
$G$-manifold. Assume the action of $G$ on $M$ is
almost free. Then the orbit space $M/G$ is a real analytic 
(resp. smooth) orbifold.
\end{theorem} 

\begin{proof} We prove the real analytic case, the smooth case is similar.
Let $\tilde{x}\in M$. It follows from the real analytic slice theorem that there is a
slice $U$ at  $\tilde{x}$ and a $G_{\tilde{x}}$-equivariant real analytic
diffeomorphism $U\to{\mathbb{R}}^m$ where $G_{\tilde{x}}$ acts orthogonally
on ${\mathbb{R}}^m$, and $m=\dim{M}-\dim(G/G_{\tilde{x}})
=\dim{M}-\dim{G}$.
More precisely,  we may consider $U$ as the
normal space ${\rm N}_{\tilde{x}}={\rm T}_{\tilde{x}}(M)/{\rm T}_{\tilde{x}}(G\tilde{x})$ 
of the orbit $G\tilde{x}$ of $\tilde{x}$ at $\tilde{x}$. The orbifold charts of $M/G$ are then
$(U, G_{\tilde{x}},\pi)$ where $\pi\colon U\to U/G_{\tilde{x}}$ is the natural projection.

Since $G$ acts properly on $M$, it follows that $M/G$ is a Hausdorff
space (Theorem 1.2.9 in \cite{Pa2}) and it also follows that $M/G$ is
paracompact.

We check that Condition $3b$ of Definition \ref{ensin} holds: Let $x,y\in 
M/G$ and let $(U, G_{\tilde{x}},\pi)$ and $(V, G_{\tilde{y}}, \pi)$ be charts such that
$\pi(\tilde{x})=x$ and $\pi(\tilde{y})=y$. 
Let $z\in \pi(U)\cap \pi(V)$. Then there
is $\tilde{z}\in U$ such that $\pi(\tilde{z})=z$. For some $g\in G$,
$g\tilde{z}\in V$. Clearly, $\pi(g\tilde{z})=z$. Consider $U$ and $V$ as
$G_{\tilde{x}}$- and $G_{\tilde{y}}$-spaces, respectively.  We can now
apply the slice theorem for those spaces. Thus, 
let ${\rm N}'_{\tilde{z}}$ 
and ${\rm N}'_{g\tilde{z}}$ be
the normal spaces of the orbits $(G_{\tilde{x}}){\tilde{z}}$ at $\tilde{z}$ in $U$ and 
 $(G_{\tilde{y}}){g\tilde{z}}$ at $g\tilde{z}$ in $V$, respectively.

Let ${\rm N}_{\tilde{z}}$ 
and ${\rm N}_{g\tilde{z}}$ be
the normal spaces of the orbits $G{\tilde{z}}$ at $\tilde{z}$ and 
 $G{g\tilde{z}}=G{\tilde{z}}$ at $g\tilde{z}$, respectively. 
Then  ${\rm N}_{\tilde{z}}$  is $G_{\tilde{z}}$-equivariantly isomorphic
to ${\rm N}'_{\tilde{z}}$ and ${\rm N}_{g\tilde{z}}$ is 
$G_{g\tilde{z}}$-equivariantly isomorphic to ${\rm N}'_{g\tilde{z}}$.
There now are real analytic embeddings
$\lambda_U\colon {\rm N}_{\tilde{z}}\cong
{\rm N}'_{\tilde{z}}\hookrightarrow U$
and $\lambda_V\colon  {\rm N}_{g\tilde{z}}\cong
{\rm N}'_{g\tilde{z}}\hookrightarrow V$, where
$\lambda_U$ is $G_{\tilde{z}}$-equivariant and 
$\lambda_V$ is $G_{g\tilde{z}}$-equivariant.
Now, $G_{g\tilde{z}}=gG_{\tilde{z}}g^{-1}$ and
$\theta\colon  G_{\tilde{z}}\to gG_{\tilde{z}}g^{-1}$, 
$\theta(h)=ghg^{-1}$, is an isomorphism. The claim follows, since the real
analytic diffeomorphism $g\colon M\to M$ induces a real
analytic $\theta$-equivariant isomorphism ${\rm N}_{\tilde{z}}\to {\rm N}_{g\tilde{z}}$.
\end{proof}

Orbit spaces of real analytic
manifolds by real analytic proper almost free actions of
Lie groups are  called real analytic {\it quotient orbifolds}. 

Assume $X$ is an $n$-dimensional  smooth orbifold, i.e., an orbifold differentiable of
degree ${\rm C}^\infty$. Assume that $X$ is {\it reduced}. This means that
all the groups in the definition of an orbifold chart act effectively.
Then it is well-known that $X$ is a quotient orbifold. More precisely,
there exists a smooth manifold $M$ on which the orthogonal group
$O(n)$ acts smoothly, effectively and almost freely such that the
orbit space $M/O(n)$ is smoothly diffeomorphic to $X$ as an orbifold.

It would be nice to have a corresponding result for real analytic reduced orbifolds.
The proof of the smooth case is based on the use of the frame bundle over
$X$, hence it makes use of a smooth Riemannian metric on $X$.  The proof can not
be applied to the real analytic case since, as far as we know, it is not known
how to construct real analytic Riemannian metrics for real analytic orbifolds
unless the orbifold already is known to be a quotient. The smooth result
gives a smooth quotient orbifold smoothly diffeomorphic to a given real analytic
reduced orbifold and, by using results from equivariant differential topology,
that quotient orbifold can be given a real analytic structure. 
The given real analytic reduced orbifold and the constructed real analytic quotient 
orbifold  are then smoothly diffeomorphic. To obtain a real analytic
diffeomorphism, one would need to be able to approximate smooth
orbifold maps by real analytic ones in some topology resembling the
Whitney topology for maps between real analytic manifolds.

\section{Suborbifolds}
\label{suborbi}

\noindent In this section we briefly discuss suborbifolds. Suborbifolds will
only be used in Section \ref{uniformi}.

\begin{definition}
\label{neat}
Let $X$ be an $n$-dimensional real analytic (resp. smooth) orbifold.
We say that
$Y$ is an  $m$-dimensional real analytic (resp. smooth)
{\it suborbifold} of $X$ if the following hold:
\begin{enumerate}
\item  $Y$ is a subset of $X$ equipped with the subspace topology.

\item For each $y\in Y$ and for each neighborhood $W$ of $y$ in $X$, there is
an orbifold chart $(\tilde{U}, G, \varphi)$ for $X$ with  $U=\varphi(\tilde{U})\subset W$,
and a subset $\tilde{V}$ of $\tilde{U}$ such that
$(\tilde{V}, G,\varphi\vert \tilde{V})$ is an 
$m$-dimensional orbifold chart for $Y$ and  $y\in \varphi(\tilde{V})$.

\item $\varphi(\tilde{V})=U\cap Y$.
\end{enumerate}
\end{definition}

Notice that there exists an alternative definition of a suborbifold. That definition
allows the group of the suborbifold chart to be any subgroup of the
group of the corresponding orbifold chart and
gives a strictly larger family of suborbifolds
than ours.

\begin{theorem} 
\label{apusub}
Let $X$ be a real analytic (resp. smooth) quotient orbifold and let
$Y$ be a real analytic (resp. smooth) suborbifold of $X$. Then $Y$ is a quotient orbifold.
\end{theorem}

\begin{proof} We prove the real analytic case, the smooth case is similar.
Since $X$ is a real analytic quotient orbifold, there is a
Lie group $G$ and  a real analytic manifold $M$
on which $G$ acts via a proper real analytic almost free action such that
the orbit space $M/G$ equals $X$.  Let $\pi\colon M\to M/G$ be the natural
projection. Then $\pi^{-1}(Y)$ is a $G$-invariant subset of $M$ on which
$G$ acts properly and almost freely. It remains to show that $\pi^{-1}(Y)$
is a real analytic submanifold of $M$.

Let $y\in Y$ and let $W$ be a neighborhood of $y$ in $M/G$. Then $M/G$ has
an orbifold chart $(\tilde{U}, G_{\tilde{y}}, \pi)$ where $\pi(\tilde{y})=y$,
$\tilde{U}$ is a linear slice at $\tilde{y}$ and $U=\pi(\tilde{U})\subset W$.
Moreover, $\tilde{U}$ has a subset $\tilde{V}$ such that 
$(\tilde{V}, G_{\tilde{y}}, \pi)$ is an orbifold chart for $Y$, $y\in \pi(
\tilde{V})$ and $\pi(\tilde{V})=U\cap Y$.

Clearly, $\tilde{V}\subset \pi^{-1}(Y)\cap \tilde{U}$. Since $\pi(\tilde{V})=U\cap Y$,
it follows that $\pi(\pi^{-1}(Y)\cap \tilde{U})\subset \pi(\tilde{V})$. Thus
$\pi^{-1}(Y)\cap \tilde{U}\subset G\tilde{V}$. Since $\tilde{U}$ is a slice at
$\tilde{y}$, it follows that $g\tilde{U}\cap \tilde{U}=\emptyset$, 
for every $g\in G\setminus G_{\tilde{y}}$. Hence 
$\pi^{-1}(Y)\cap \tilde{U}\subset G_{\tilde{y}}\tilde{V}=\tilde{V}$.
Thus $\tilde{V}= \pi^{-1}(Y)\cap \tilde{U}$ and we see that every orbifold chart of
$Y$ is of the form $(\pi^{-1}(Y)\cap \tilde{U}, G_{\tilde{y}}, \pi)$  where
$(\tilde{U}, G_{\tilde{y}}, \pi)$ is an orbifold chart of $M/G$.

Now, $\pi^{-1}(Y)$ can be written as a union of {\it twisted products}
$G\times_{G_{\tilde{x}}}(\pi^{-1}(Y)\cap\tilde{U})=
(G\times_{G_{\tilde{x}}} \tilde{U})\cap \pi^{-1}(Y)$, where the
$(\tilde{U}, G_{\tilde{x}}, \pi)$ are orbifold charts of $M/G$. Since 
each $\pi^{-1}(Y)\cap \tilde{U}$ is a real analytic $G_{\tilde{x}}$-manifold,
it follows that each twisted product is a proper real analytic $G$-manifold. 
Assume $\tilde{y}\in G\times_{G_{\tilde{x}_i}}(\pi^{-1}(Y)\cap\tilde{U}_i)$,
for $i=1,2$. Then there is a slice $\tilde{U}$ at $\tilde{y}$ such that
$G\times_{G_{\tilde{y}}}\tilde{U}\subset
(G\times_{G_{\tilde{x}_1}}\tilde{U}_1)\cap (G\times_{G_{\tilde{x}_2}}\tilde{U}_2)$.
Thus 
$G\times_{G_{\tilde{y}}}(\pi^{-1}(Y)\cap\tilde{U})\subset
G\times_{G_{\tilde{x}_1}}(\pi^{-1}(Y)\cap\tilde{U}_1)
\cap G\times_{G_{\tilde{x}_2}}(\pi^{-1}(Y)\cap\tilde{U}_2)$ and it
follows that $\pi^{-1}(Y)$ is a real analytic manifold.  The group $G$
acts real analytically on $\pi^{-1}(Y)$, since the action is just the 
restriction of the action on $M$.
\end{proof}

\section{Semianalytic and subanalytic subsets of real analytic manifolds}
\label{mani}

\noindent Semianalytic sets were first studied by S. {\L}ojasiewicz, see
\cite{Lo}.  Subanalytic sets are a generalization of semianalytic sets,
considered by Hironaka in \cite{Hi1}. We recall the definitions of
semianalytic and subanalytic sets.  See \cite{BM}, \cite{Hi1} and
\cite{Lo} for their basic properties. 
Some elementary properties are also proved in \cite{Ka}.

Let $M$ be a real analytic manifold and let $U$ be an open subset of
$M$. We denote by ${\rm C}^\omega(U)$ the set of all real analytic maps
$U\to {\mathbb{R}}$. The smallest family of subsets of $U$ containing
all the sets $\{ x\in U\mid f(x)>0\}$, where $f\in {\rm C}^\omega(U)$, which
is stable under finite intersection, finite union and complement, is denoted
by $S({\rm C}^\omega(U))$. 

\begin{definition} Let $M$ be a real analytic manifold. A subset $A$ of
$M$ is called {\it semianalytic} if every point $x\in M$ has a neighborhood $U$
such that $A\cap U\in S({\rm C}^\omega(U))$.
\end{definition}

Finite unions and finite
intersections of semianalytic sets are semianalytic and the complement
of a semianalytic set is semianalytic (Remark 2.2 in \cite{Hi1}).
Also, the closure and the interior of a semianalytic set are semianalytic
(Corollary 2.8 in \cite{BM}). Every connected component of a 
semianalytic set is semianalytic, and the family of the connected components of a
semianalytic set is locally finite (Corollary 2.7 in \cite{BM}).  
Moreover, the union of any
collection of connected components of a semianalytic set is semianalytic.

Let $A\subset M$. We say that $A$ is a {\it projection of a semianalytic set}
if there exists a real analytic manifold $N$ and a semianalytic subset
$B$ of $M\times N$ such that $A=p(B)$, where $p\colon M\times N\to M$
is the projection. We call $B$ {\it relatively compact}, if the closure ${\overline{B}}$
is compact.

\begin{definition}
\label{sub}
Let $M$ be a real analytic manifold. A subset $A$ of $M$ is called
{\it subanalytic} if every point $x\in M$ has a neighborhood $U$
such that $A\cap U$ is a projection of a relatively compact
semianalytic set.
\end{definition}

Every semianalytic set is subanalytic (Proposition 3.4.  in \cite{Hi1}).
For subanalytic sets that are not
semianalytic, see Example 2 on p. 453 in \cite{Hi1} and Examples 1 and 2
on p. 134--135 in \cite{Lo}.
Finite unions and and finite
intersections of subanalytic sets are subanalytic, and the complement of a
subanalytic set is subanalytic (Proposition 3.2 in \cite{Hi1}).
The closure and thus also the interior of a
subanalytic set is subanalytic  (Corollary 3.7.9 in \cite{Hi1}).
Connected components of a subanalytic set
are subanalytic (Corollary 3.7.10 in \cite{Hi1})
and  the family of the connected components of a subanalytic
set is locally finite (Proposition 3.6 in \cite{Hi1}). The union of
any collection of connected components of a subanalytic set
is subanalytic.

\begin{definition}
\label{submaps}
Let $M$ and $N$ be real analytic manifolds and let $A$ be a semianalytic
(subanalytic) subset of $M$. A continuous map
$f\colon A\to N$ is called semianalytic (subanalytic) if its graph ${\rm Gr}(f)$
is a semianalytic (subanalytic) subset of $M\times N$.
\end{definition}

Clearly, every real analytic map $M\to N$ is semianalytic and every semianalytic
map is subanalytic.

We mention yet another important property of subanalytic sets:

\begin{theorem}
\label{hiro}
Let $M$ and $N$ be real analytic manifolds, let $A$ be a subanalytic
subset of $M$ and let $B$ be a subanalytic subset of $N$. Let 
$f\colon M\to N$ be a real analytic map.  Then $f^{-1}(B)$ is a subanalytic
subset of $M$. If, in addition, $f$ is proper, then $f(A)$ is a subanalytic subset of $N$.
\end{theorem}

\begin{proof}
Proposition 3.8 in \cite{Hi1}.
\end{proof}

\begin{cor}
\label{mina}
Let $M$ and $N$ be real analytic manifolds, let $A$ be a subanalytic
subset of $M$ and let $B$ be a subanalytic subset of $N$. Let 
$f\colon M\to N$ be a subanalytic map.  Then $f^{-1}(B)$ is a subanalytic
subset of $M$. If, in addition, $f$ is proper, then $f(A)$ is a subanalytic subset of $N$.
\end{cor}

\begin{proof}
Corollary 4.23 in \cite{Ka}.
\end{proof}

Notice that both in Proposition  3.8 in \cite{Hi1} and in Corollary 4.23 in
\cite{Ka}, the statements about the inverse image $f^{-1}(B)$ are for
proper maps. However, properness is not used in the proofs and the
satements also hold when $f$ is not proper.

\section{Semianalytic and subanalytic subsets of orbifolds}
\label{semisubdefi}

\noindent In this section we define semianalytic and subanalytic
subsets for orbifolds and show that the basic properties
that hold in the manifold case, mentioned in Section \ref{mani},
also hold in the orbifold case.

\begin{definition}
\label{maar1}
Let $X$ be a real analytic orbifold. A subset $A$ of $X$ is called
{\it semianalytic}  ({\it subanalytic}) if  for every  point $x$ of $X$ there is an 
orbifold chart $(\tilde{V},G,\varphi)$ of $X$ such that
$x\in V=\varphi(\tilde{V})$ and $\varphi^{-1}(A\cap V)$ is a 
semianalytic  (subanalytic) subset of $\tilde{V}$. 
\end{definition}

It is clear from Definition \ref{maar1} that every semianalytic subset of $X$
is also subanalytic.

Let $M$ be a real analytic manifold. The trivial group
acts on each chart of $M$,  and it follows that $M$ is a
real analytic orbifold. Thus a subset of $M$ can be semianalytic
(subanalytic) either in the manifold sense or in the orbifold sense.
In fact, for a real analytic manifold, the two definitions of semianalyticity
(subanalyticity) are equivalent:

\begin{theorem}
\label{samat}
Let $M$ be a real analytic manifold and let $A\subset M$. Then
$A$ is semianalytic (subanalytic) in the manifold sense if and only if it
is semianalytic (subanalytic) in the orbifold sense.
\end{theorem}

\begin{proof}
Assume first that $A$ is semianalytic (subanalytic) in the manifold
sense. Let $x\in M$ and let $(U,\varphi)$ be a chart of $M$ such that
$x\in \varphi(U)$. Then $\varphi(U)$ is an open subset of $M$
and $A\cap\varphi(U)$ is semianalytic (subanalytic) in $\varphi(U)$.
Since $\varphi\colon U\to \varphi(U)$ is a real analytic diffeomorphism,
it follows that
$\varphi^{-1}(A\cap \varphi(U))$ is a semianalytic (subanalytic)
subset of $U$. Thus $A$ is semianalytic (subanalytic) in the
orbifold sense.

Assume then that $A$ is semianalytic (subanalytic) in the orbifold sense.
Let $x\in M$. Then $M$ has a chart $(U,\varphi)$ such that $x\in
\varphi(U)$ and $\varphi^{-1}(A\cap\varphi(U))$ is semianalytic
(subanalytic) in $U$. It follows that $A\cap\varphi(U)$ is semianalytic
(subanalytic) in $\varphi(U)$ in the manifold sense. 
Thus $\varphi(U)$ is a neighborhood
of $x$ such that $A\cap\varphi(U)$ is semianalytic (subanalytic)
in $\varphi(U)$, and it follows that $A$ is a semianalytic (subanalytic)
in the manifold sense.
\end{proof}

\begin{lemma}
\label{apu}
Let $X$ be a real analytic orbifold and let $A$ be a semianalytic
(subanalytic) subset of $X$. Let $(\tilde{V}, G,\varphi)$ be an
orbifold chart such that $\varphi^{-1}(A\cap V)$ is a semianalytic
(subanalytic) subset of $\tilde{V}$, where $V=\varphi(\tilde{V})$.
Let $(\lambda, h)\colon (\tilde{U}, H,\psi)\to (\tilde{V}, G,\varphi)$ be an
embedding between two orbifold charts and let $U=\psi(\tilde{U})$. 
Then $\psi^{-1}(A\cap U)=
\lambda^{-1}( \varphi^{-1}(A\cap V))$ is a semianalytic (subanalytic)
subset of $\tilde{U}$.
\end{lemma}

\begin{proof}
The image $\lambda(\tilde{U})$ is open in $\tilde{V}$. 
 Then $\lambda(\tilde{U})\cap \varphi^{-1}(A\cap V)$
is semianalytic (subanalytic) in $\lambda(\tilde{U})$. Since
$\lambda$ is a real analytic diffeomorphism onto the image
$\lambda(\tilde{U})$, it follows that 
$\psi^{-1}(A\cap U)=
\lambda^{-1}( \varphi^{-1}(A\cap V))$
is a semianalytic (subanalytic) subset of $\tilde{U}$.
\end{proof}

The following theorem lists some of the basic properties of
semianalytic and subanalytic subsets of real analytic orbifolds. 
All properties follow from the
corresponding properties of semianalytic and subanalytic subsets
of real analytic manifolds.

\begin{theorem}
\label{basicprop}
Let $X$ be a real analytic orbifold. Then
{\begin{enumerate}
\item Finite unions of semianalytic (subanalytic) subsets of $X$
are semianalytic (subanalytic).
\item Finite intersections of semianalytic (subanalytic) subsets of $X$
are semianalytic (subanalytic).
\item A complement of a semianalytic (subanalytic) subset 
of $X$ is semianalytic (subanalytic).
\item A closure of a semianalytic (subanalytic) subset 
of $X$ is semianalytic (subanalytic).
\item An interior of a semianalytic (subanalytic) subset of $X$ is
semianalytic (subanalytic).
\item Every connected component of a semianalytic (subanalytic)
set is semianalytic (subanalytic).
\item The family of the connected components of a semianalytic
(subanalytic) set is locally finite.
\end{enumerate}}
\end{theorem}

\begin{proof}
Let $A_i$, $i=1,\ldots, n$, be semianalytic (subanalytic) subsets 
of $X$ and let $x\in X$. For every $i$, there is an orbifold chart
$(\tilde{V}_i, G_i,\varphi_i)$, where $x\in V_i=\varphi_i(\tilde{V}_i)$,
such that $\varphi^{-1}_i(A_i\cap V_i)$ is semianalytic (subanalytic)
in $\tilde{V}_i$. Let $(\tilde{V}, G, \varphi)$  be an orbifold chart,
where $x\in V=\varphi(\tilde{V})$, and such that there is an embedding
$(\lambda_i, h_i)\colon (\tilde{V}, G, \varphi)\to (\tilde{V}_i, G_i,\varphi_i)$,
for every $i=1,\ldots, n$. By Lemma \ref{apu}, $\varphi^{-1}(A_i\cap V)$
is semianalytic (subanalytic) in $\tilde{V}$, for every $i=1,\ldots, n$.
But then
$$
\varphi^{-1}(\bigcup_{i=1}^nA_i\cap V)=\bigcup_{i=1}^n\varphi^{-1}(A_i\cap V)
$$
is semianalytic (subanalytic) in $\tilde{V}$. Thus $\bigcup_{i=1}^nA_i$
is semianalytic (subanalytic) in $X$ and Claim (1) follows. The proof of
Claim (2) is similar.

Let $A$ be a semianalytic (subanalytic) subset of $X$ and let $x\in X$.
Then there is an orbifold chart $(\tilde{V}, G, \varphi)$ of $X$ such that
$x\in V=\varphi(\tilde{V})$ and $\varphi^{-1}(A\cap V)$ is semianalytic
(subanalytic) in $\tilde{V}$. But then $\varphi^{-1}( (X\setminus A)\cap V)=
\tilde{V}\setminus \varphi^{-1}(A\cap V)$ is semianalytic (subanalytic) in
$\tilde{V}$. Consequently, $X\setminus A$ is semianalytic (subanalytic)
in $X$ and Claim (3) follows.

To prove Claim (4), let $A$ be a semianalytic (subanalytic) subset
of $X$, and let $x\in X$. Again, there exists an orbifold chart
$(\tilde{V}, G, \varphi)$ such that $x\in V=\varphi(\tilde{V})$ and
$\varphi^{-1}(A\cap V)$ is a semianalytic (subanalytic) subset
of $\tilde{V}$.  But then the closure $\overline{\varphi^{-1}(A\cap V)}$
is also a semianalytic (subanalytic) subset of $\tilde{V}$. Since
$\varphi$ can be considered as the natural projection $\tilde{V}\to
\tilde{V}/G\approx V$, it follows that $\varphi^{-1}(\overline{A}\cap V)
=\overline{\varphi^{-1}(A\cap V)}$. Thus the closure
$\overline{A}$ is a semianalytic (subanalytic) subset of $X$.

Since the interior of a set is the complement of the closure of its
complement, Claim (5) follows from Claims (3) and (4).

Let $A$ be a semianalytic (subanalytic) subset of $X$ and let 
$A_0$ be a connected component of $A$. Let $x\in X$ and let
$(\tilde{V},G,\varphi)$ be an orbifold chart such that $x\in \varphi(\tilde{V})=
V$ and $\varphi^{-1}(A\cap V)$ is a semianalytic (subanalytic) subset
of $\tilde{V}$. Then every connected component of $\varphi^{-1}(A\cap V)$ is
semianalytic (subanalytic) in $\tilde{V}$ and $\varphi^{-1}(A_0\cap V)$
is a union of some connected components of  $\varphi^{-1}(A\cap V)$.
Thus $\varphi^{-1}(A_0\cap V)$ is a semianalytic (subanalytic) subset
on $\tilde{V}$ and it follows that $A_0$ is semianalytic (subanalytic).
This proves Claim (6).
Let then $\tilde{x}\in\tilde{V}$ be such that $x=\varphi(\tilde{x})$. 
Then $\tilde{x}$ has a neighborhood $U$ in $\tilde{V}$ such that
$U$ intersects only finitely many of the connected components 
of $\varphi^{-1}(A\cap V)$. But then $\varphi(U)$ is a neighborhood of
$x$ that intersects only finitely many connected components of $A$.
Thus Claim (7) follows. 
\end{proof}

\section{Semianalytic and subanalytic maps in the orbifold case}
\label{kuvausdef}

\begin{definition}
\label{maar2}
A  map $f\colon X\to Y$ between two real analytic orbifolds is 
called  {\it real analytic}
(resp. {\it semianalytic} or  {\it subanalytic})
if for every $x\in X$ there are orbifold charts $(\tilde{U}, G,\varphi)$
and $(\tilde{V}, H,\psi)$, where $x\in U$ and $f(x)\in V$,  a homomorphism
$\theta\colon G\to H$ and a $\theta$-equivariant real analytic
(resp. semianalytic or subanalytic) map $\tilde{f}\colon
\tilde{U}\to\tilde{V}$ making the following diagram commute:

$$\begin{CD}
\tilde{U}
@>\tilde{f}>>  \tilde{V}\\
@VVV     @VVV\\
\tilde{U}/G
@>{}>>  \tilde{V}/H\\
@VVV     @VVV\\
 U      
 @>f\vert U>>   V\\
\end{CD}.
$$
\end{definition}

It follows  that a real analytic (resp. semianalytic or subanalytic)
map between two orbifolds is automatically continuous. Notice that a continuous
map $f\colon X\to Y$ does not necessarily have continuous local lifts  $\tilde{f}$ making the
diagram in \ref{maar2} commute. A continuous map that does have such local
lifts is called an {\it orbifold map}. Thus, in particular, all real analytic, semianalytic
and subanalytic maps between orbifolds are orbifold maps.

A real analytic map $f\colon X\to Y$ is called a real analytic {\it  diffeomorphism}, if it 
is a bijection with  a real analytic inverse map.

\begin{lemma}
\label{ennengraafia}
Let $M$ be a real analytic manifold and let $G$ be a finite group
acting real analytically on $M$. Let $I\subset G$ and let $A_I$ be the
subset of $M$ consisting of points $x$ such that $gx=x$ if and only if
$g\in I$. Then $A_I$ is a semianalytic subset of $M$.
 \end{lemma}
 
 \begin{proof}
 Notice that $A_I\not=\emptyset$ if and only if $I$ equals the
 isotropy subgroup of some point in $M$. For every $g\in G$,
 let $A_g=\{ x\in M\mid gx=x\}$. Then
 $$
 A_I=(\bigcap_{g\in I}A_g)\cap(\bigcap_{g\in G\setminus I}M\setminus A_g).
 $$
 Let $e\colon M\to {\mathbb{R}}^n$ be a proper real analytic embedding in
 some euclidean space, and let
 $$
 f_g\colon M\to{\mathbb{R}},\,\,\, x\mapsto \Vert e(gx)-e(x)\Vert^2.
 $$
 Then $f_g$ is a real analytic map and $A_g=f^{-1}_g(0)$.
 Thus $A_g$ is semianalytic. By the complement rule, also $M\setminus A_g$ is
 semianalytic. It follows that $A_I$ is semianalytic as an intersection of
 finitely many semianalytic sets.
 \end{proof}

Compare the following theorem to Definition \ref{submaps}:

\begin{theorem}
\label{graafi}
Let $X$ and $Y$ be real analytic orbifolds and let $f\colon X\to Y$ be a
continuous orbifold map. Then
 $f$ is semianalytic (subanalytic) if and only if the graph ${\rm Gr}(f)$ is a
semianalytic (subanalytic) subset of $X\times Y$. 
\end{theorem}

\begin{proof}
We prove the semianalytic case. The subanalytic case is similar.

Assume first that $f$ is semianalytic. We show that the graph ${\rm Gr}(f)$
is a semianalytic subset of $X\times Y$.
Let $(x,y)\in X\times Y$.
If $y\not= f(x)$, then $y$ and $f(x)$ have disjoint neighborhoods. Thus
$(x,y)$ has a neighborhood that does not intersect ${\rm Gr}(f)$ and,
consequently, there is nothing to prove. Therefore, assume
$y=f(x)$.  Then there are orbifold charts $(\tilde{V}, G,\varphi)$
and $(\tilde{U}, H,\psi)$ of $X$ and $Y$, respectively, such that
$x\in V=\varphi(\tilde{V})$ and $y\in U=\psi(\tilde{U})$ and the
restriction $f\vert V$ has a semianalytic  equivariant lift
$\tilde{f}\colon \tilde{V}\to \tilde{U}$.  Then the graph
${\rm Gr}(\tilde{f})$ is a semianalytic subset of $\tilde{V}\times\tilde{U}$.
Since $\tilde{f}$ is semianalytic and each $h\in H$ is a real
analytic diffeomorphism of $\tilde{U}$, it follows that also
the graphs ${\rm Gr}(h\circ \tilde{f})$, $h\in H$, are semianalytic.
Thus
{\begin{equation}
(\varphi\times\psi)^{-1}( {\rm Gr}(f)\cap (V\times U))
=
\bigcup_{h\in H}{\rm Gr}(h\circ \tilde{f})
\end{equation}}
is semianalytic as a finite union of semianalytic sets. It follows that
${\rm Gr}(f)$ is semianalytic in $X\times Y$.

Assume then that the graph ${\rm Gr}(f)$ is semianalytic. We show that
$f$ is semianalytic. Let $x\in X$, and let $(\tilde{V},G,\varphi)$ and $(\tilde{U},H,\psi)$
be orbifold charts of $X$ and $Y$, respectively, where $x\in V=
\varphi(\tilde{V})$, $f(x)\in U=\psi(\tilde{U})$, and such that there is
a continuous equivariant lift $\tilde{f}\colon\tilde{V}\to \tilde{U}$ of
$f\vert V$. Choosing $V$ and $U$ to be sufficiently small, we may
assume  that $(\varphi\times\psi)^{-1}({\rm Gr}(f)\cap (V\times U))$
is a semianalytic subset of $\tilde{V}\times\tilde{U}$. Now,
equation $(1)$ holds and
we will show that ${\rm Gr}(\tilde{f})$ is a semianalytic subset of
$\tilde{V}\times\tilde{U}$. 

For any  subset
$I$ of $H$, let $A_I$ be the set of points $(z,w)\in \tilde{V}\times
\tilde{U}$ such that $(z,w)=(z,hw)$ if and only if $h\in I$.
By Lemma \ref{ennengraafia},  each
$A_I$ is semianalytic in $\tilde{V}\times\tilde{U}$. It follows that
each intersection 
$$
B_I=(\varphi\times\psi)^{-1}( {\rm Gr}(f)\cap (V\times U))\cap A_I
$$
is semianalytic.  Now,
$
{\rm Gr}(\tilde{f})\cap A_I
$
is a union of some connected components of $B_I$ that form a locally finite
family in $\tilde{V}\times\tilde{U}$. Since connected
components of semianalytic sets are semianalytic, it follows that
${\rm Gr}(\tilde{f})\cap A_I$ is semianalytic
for every $I\subset H$. Therefore,
$$
{\rm Gr}(\tilde{f})=\bigcup_{I\subset H} ({\rm Gr}(\tilde{f})\cap A_I)
$$
is semianalytic as a finite union of semianalytic sets. It follows that
$f$ is semianalytic.
\end{proof}

\section{Elementary properties}
\label{semi}

\noindent In this section we present some basic  properties of semianalytic sets
and maps.

\begin{lemma} 
\label{semi1}
Let $X$ be a real analytic orbifold. Let $A$ be a semianalytic (subanalytic)
subset of $X$ and let $W$ be an open subset of $X$. Then $A\cap W$ is a
semianalytic (subanalytic) subset of $W$.
\end{lemma}

\begin{proof}
The claim follows easily from Lemma \ref{apu}.
\end{proof}

\begin{lemma}
\label{neighb}
Let $X$ be a real analytic orbifold. Then a subset $A$ of $X$ is 
semianalytic (subanalytic)
if and only if every point $x\in X$ has an open neighborhood $U$ such that
$A\cap U$ is semianalytic (subanalytic) in $U$.
\end{lemma}

\begin{proof}
If $A$ is a semianalytic (subanalytic)  subset of $X$, we can choose
$U$ to equal $X$ for all $x\in X$.
Assume then that $A$ is a subset of $X$ having the property that every
$x\in X$ has an open neighborhood $U_x$  for which $A\cap U_x$ is
semianalytic (subanalytic) in $U_x$. Let $x\in X$. 
Then there exists an orbifod chart $(\tilde{V}, G,\varphi)$ with 
$x\in\varphi(\tilde{V})=V\subset U_x$ such that $\varphi^{-1}(A\cap V)$ is
semianalytic (subanalytic) in $\tilde{V}$. Since this holds for any
$x\in X$, it follows that $A$ is seminalytic (subanalytic) in $X$.
\end{proof}

Compare the following proposition to Definition \ref{maar1}.
 
 \begin{prop}
 \label{huhhuh}
 Let $X$ be a real analytic orbifold and let $A\subset X$.  
 Then $A$ is  a semianalytic (subanalytic) subset of $X$
 if and only if $\varphi^{-1}(A\cap V)$ is a semianalytic
 (subanalytic) subset of $\tilde{V}$ for any orbifold chart
 $(\tilde{V}, G, \varphi)$  of $X$,
 where $V=\varphi(\tilde{V})$.
 \end{prop}
 
 \begin{proof}
 We prove the subanalytic case. The semianalytic case is similar,
 just the word "subanalytic" should everywhere be replaced by the word
 "semianalytic". If $\varphi^{-1}(A\cap V)$ is a subanalytic subset of
 $\tilde{V}$ for every orbifold chart $(\tilde{V}, G,\varphi)$ of $X$, then it
 follows from Definition \ref{maar1} that $A$ is subanalytic.
 
 Assume then that $A$ is subanalytic. Let $(\tilde{V}, G,\varphi)$
 be an orbifold chart of $X$, and let $V=\varphi(\tilde{V})$.
 We have to show that $\varphi^{-1}(A\cap V)$ is a subanalytic subset of
 $\tilde{V}$. Let $\tilde{y}\in \tilde{V}$. Then $y=\varphi(\tilde{y})\in
 V$. Since $A$ is a subanalytic subset of $X$, there is an orbifold
 chart $(\tilde{W},H,\psi)$ such that $y\in \psi(\tilde{W})=W$ and
 $\psi^{-1}(A\cap W)$ is a subanalytic subset of $\tilde{W}$. 
 Choosing $W$ to be sufficiently small and using  Lemma
 \ref{apu}, we may assume that there is an embedding
 $(\lambda, h)\colon (\tilde{W},H,\psi)\to (\tilde{V}, G, \varphi)$.
 Since $\psi^{-1}(A\cap W)$ is subanalytic in $\tilde{W}$, it
 follows that $\varphi^{-1}(A\cap V)\cap \lambda(\tilde{W})=
 \lambda(\psi^{-1}(A\cap W))$ is subanalytic in $\lambda(\tilde{W})$.
 Thus, for some $g\in G$,  $g\lambda(\tilde{W})$ is a neighborhood
  of $\tilde{y}$ such that 
 $\varphi^{-1}(A\cap V)\cap g\lambda(\tilde{W})$ is subanalytic
 in $g\lambda(\tilde{W})$. Since $\tilde{y}$ was an arbitrary point of
 $\tilde{V}$, it follows  that $\varphi^{-1}(A\cap V)$ 
 is subanalytic in $\tilde{V}$.
 \end{proof}

\begin{lemma}
\label{locfin}
Let $X$ be a real analytic orbifold and let $\{ A_i\}_{i\in I}$ be a locally
finite family of semianalytic (subanalytic) subsets of $X$. Then $\bigcup_{i\in I}A_i$
is a semianalytic (subanalytic) subset of $X$.
\end{lemma}

\begin{proof}
Let $x\in X$. Then $x$ has an open neighborhood $U$ such that
$U\cap A_i=\emptyset$, except for finitely many indices $i\in I$, say
$i=1,\ldots, n$. The finite union $\bigcup_{i=1}^n A_i$ is 
semianalytic (subanalytic) in $X$. 
By Lemma \ref{semi1}, $(\bigcup_{i\in I}A_i)\cap U=
(\bigcup_{i=1}^nA_i)\cap U$
is semianalytic (subanalytic) in $U$. Since $x$ was chosen arbitrarily, it follows from
Lemma \ref{neighb} that $\bigcup_{i\in I}A_i$ is semianalytic (subanalytic) in $X$.
\end{proof}

\begin{lemma}
\label{semi2}
Let $X$ and $Y$ be real analytic orbifolds, $A$ a semianalytic 
(subanalytic) subset of $X$
and $B$ a semianalytic (subanalytic) subset of $Y$. Then $A\times B$ is a 
semianalytic (subanalytic) subset 
of $X\times Y$.
\end{lemma}

\begin{proof}
Let $(x,y)\in X\times Y$. Then there are orbifold charts
$(\tilde{V}, G,\varphi)$ and $(\tilde{U},H,\psi)$ of $X$ and $Y$,
such that $x\in \varphi(\tilde{V})=V$ and $y\in\psi(
\tilde{U})=U$ and $\varphi^{-1}(A\cap V)$ and $\psi^{-1}(B\cap U)$ are semianalytic
(subanalytic)
in $\tilde{V}$ and $\tilde{U}$, respectively. Since the product of semianalytic
(subanalytic) sets is semianalytic (subanalytic) in the manifold case
(Lemmas 4.5 and 4.17 in \cite{Ka}),
the  set $(\varphi\times\psi)^{-1}((A\times B)\cap (V\times U))=
\varphi^{-1}(A\cap V)\times\psi^{-1}(B\cap U)$ is semianalytic (subanalytic) in
$\tilde{V}\times\tilde{U}$.
\end{proof}

\begin{prop}
\label{subprod1}
Let $X_1$, $X_2$, $Y_1$ and $Y_2$ be real analytic orbifolds and
let $f_1\colon X_1\to Y_1$ and $f_2\colon X_2\to Y_2$ be 
semianalytic (subanalytic)
maps. Then the map
$$f_1\times f_2\colon X_1\times X_2\to Y_1\times Y_2,\,\,\,
(x_1,x_2)\mapsto (f_1(x_1),f_2(x_2)),
$$ 
is semianalytic (subanalytic).
\end{prop}

\begin{proof} 
%
By Theorem \ref{graafi} and Lemma \ref{semi2}, ${\rm Gr}(f_1)\times
{\rm Gr}(f_2)$ is a semianalytic (subanalytic) 
subset of $X_1\times Y_1\times X_2\times Y_2$.
Let 
$$
f\colon X_1\times Y_1\times X_2\times Y_2\to X_1\times X_2\times
Y_1\times Y_2
$$ 
be the map
changing the order of  coordinates. Since $f$ is a real 
analytic orbifold diffeomorphism, it follows that ${\rm Gr}(f_1\times f_2)=
f({\rm Gr}(f_1)\times{\rm Gr}(f_2))$ is semianalytic
(subanalytic) in $X_1\times X_2\times
Y_1\times Y_2$. The claim now follows from Theorem \ref{graafi}.
\end{proof}
 
 \begin{prop}
 \label{subprod2}
 Let $X$, $Y_1$ and $Y_2$ be real analytic orbifolds and let
 $f_1\colon X\to Y_1$ and $f_2\colon X\to Y_2$ be semianalytic
 (subanalytic) maps. Then the map 
 $$
 (f_1, f_2)\colon X\to Y_1\times Y_2,\,\,\, x\mapsto(f_1(x),f_2(x)),
 $$ 
 is semianalytic (subanalytic).
 \end{prop}
 
 \begin{proof} Since $f_1$ and $f_2$ are semianalytic (subanalytic), 
 it follows from Lemma \ref{semi2},
 that ${\rm Gr}(f_1)\times Y_2$ and ${\rm Gr}(f_2)\times Y_1$ are 
 semianalytic (subanalytic)
 subsets of $X\times Y_1\times Y_2$ and $X\times Y_2\times Y_1$, respectively.
 Let $f\colon X\times Y_2\times Y_1\to X\times Y_1\times Y_2$ be the map
 changing the order of coordinates. Then $f({\rm Gr}(f_2)\times Y_1)$ is a 
 semianalytic (subanalytic)
 subset of $X\times Y_1\times Y_2$. Therefore,
 $$
 {\rm Gr}( (f_1,f_2) )= ({\rm Gr}(f_1)\times Y_2)\cap  f({\rm Gr}(f_2)\times Y_1)
 $$
 is semianalytic (subanalytic) as an intersection of two 
semianalytic  (subanalytic) sets. It follows from
 Theorem \ref{graafi}, that $(f_1, f_2)$ is semianalytic (subanalytic).
 \end{proof}

\section{Images and inverse images}
\label{Sub}

\noindent In this section we study images and inverse images of semi- and
subanalytic sets in the orbifold case. The orbifold case is completely
analogous to the manifold case, compare Theorem \ref{kuvat} to
Corollary \ref{mina}.

\begin{lemma}
\label{semi3}
Let $X$ and $Y$ be real analytic orbifolds and let $f\colon X\to Y$ be a real
analytic map. If $B$ is a semianalytic subset of $Y$, then $f^{-1}(B)$ is a
semianalytic subset of $X$. 
\end{lemma}

\begin{proof}
Let $x\in X$. Then there are orbifold charts 
$(\tilde{V}, G,\varphi)$ and $(\tilde{U},H,\psi)$ of $X$ and $Y$,
respectively, such that $x\in V=\varphi(\tilde{V})$, $f(x)\in U=\psi(\tilde{U})$
and $f\vert V$ has a real analytic  equivariant lift $\tilde{f}\colon \tilde{V}\to\tilde{U}$.
By Proposition \ref{huhhuh},
$\psi^{-1}(B\cap U)$ is a semianalytic subset of $\tilde{U}$. Since the inverse
image of a semianalytic set is semianalytic in the manifold case
(Lemma 4.6 in \cite{Ka}), it follows that $\varphi^{-1}(f^{-1}(B)\cap V)= {\tilde{f}}^{-1}(
\psi^{-1}(B\cap U))$ is a semianalytic subset of $\tilde{V}$.
\end{proof}

 \begin{lemma}
 \label{ymp}
 Let $X$ be a real analytic orbifold, let $x\in X$ and let
 $U$ be a neighborhood of $x$. Then $x$ has a relatively
 compact neighborhood $O$ such that $\overline{O}\subset
 U$ and $O$ is a semianalytic subset of $X$.
  \end{lemma}
 
 \begin{proof}
 Let $(\tilde{V}, G, \varphi)$ be an orbifold chart such that
 $x\in \varphi(\tilde{V})=V\subset U$. Let $\tilde{x}\in \tilde{V}$
 be such that $\varphi(\tilde{x})=x$. Now, let $\tilde{O}\subset\tilde{V}$
 be a ball with $\tilde{x}$ as a center. We assume the radius of
 $\tilde{O}$ to be so small that  ${\overline{\tilde{O}}}$ is compact.
 Then $\tilde{O}$ is relatively compact and semianalytic.
Thus $O=\varphi(\tilde{O})$ is a relatively compact neighborhood
 of $x$ with $\overline{O}\subset V$. Moreover, $O$ is semianalytic,
 since $\varphi^{-1}(O)=\bigcup_{g\in G}(g\tilde{O})$ is semianalytic as a
 finite union of semianalytic sets.
 \end{proof}

\begin{theorem}
\label{kuvat}
Let $X$ and $Y$ be real analytic orbifolds and let $f\colon X\to Y$ be
a subanalytic map. If $B$ is a subanalytic subset of $Y$, then $f^{-1}(B)$
is a subanalytic subset of $X$. If, in addition, $f$ is a proper map, then the
image $f(A)$ of any subanalytic subset $A$ of $X$ is subanalytic in $Y$.
\end{theorem}

\begin{proof}
Let $B$ be a subanalytic subset of $Y$. Let $x\in X$ and let $(\tilde{V}, G,\varphi)$
and $(\tilde{U}, H,\psi)$ be orbifold charts of $X$ and $Y$, respectively, such that
$x\in V=\varphi(\tilde{V})$, $f(x)\in U=\psi(\tilde{U})$, and $f\vert V$ has a subanalytic
equivariant lift $\tilde{f}\colon \tilde{V}\to \tilde{U}$.  By Proposition \ref{huhhuh},
$\psi^{-1}(B\cap U)$ is a subanalytic subset of $\tilde{U}$. By Corollary
\ref{mina}, the inverse image
$\tilde{f}^{-1}(\psi^{-1}(B\cap U))$ is subanalytic in $\tilde{V}$.  The first claim follows,
since $\varphi^{-1}(f^{-1}(B)\cap V)=\tilde{f}^{-1}(\psi^{-1}(B\cap U))$.

Assume then that $f$ is a proper map and that $A$ is a subanalytic subset of
$X$. Remembering that orbifolds are paracompact and using Lemma \ref{ymp},
it is possible to construct a locally finite open cover $\{{O}_i\}_{i\in I}$ of $X$
such that each $O_i$ is subanalytic  and each
closure $\overline{O}_i$ is compact, $\overline{O}_i\subset
V_i$ and each restriction $f\vert\colon V_i\to U_i$ has a subanalytic equivariant
lift $\tilde{f}_i\colon \tilde{V}_i\to \tilde{U}_i$, where $(\tilde{V}_i, G_i,\varphi_i)$ and
$(\tilde{U}_i, H_i,\psi_i)$ are orbifold charts of $X$ and $Y$, respectively, and
$V_i=\varphi_i(\tilde{V}_i)$ and $U_i=\psi_i(\tilde{U}_i)$. 

Since $A\cap \overline{O}_i$ is subanalytic in $X$, it follows 
from Proposition \ref{huhhuh} that $\varphi_i^{-1}(A\cap \overline{O}_i)$ 
is subanalytic in $\tilde{V}_i$. Since $\varphi_i^{-1}(A\cap \overline{O}_i)$
is relatively compact and since $\tilde{f}_i$ is subanalytic, it follows 
from Corollary \ref{mina} that
$\tilde{f}_i(\varphi_i^{-1}(A\cap \overline{O}_i))$ is subanalytic in $\tilde{U}_i$.
But then also $h(\tilde{f}(\varphi_i^{-1}(A\cap \overline{O}_i)))$ is subanalytic
in $\tilde{U}_i$, for every $h\in H_i$. Therefore,
$$
\psi^{-1}_i( f(A\cap \overline{O}_i))
=\bigcup_{h\in H_i}h(\tilde{f}(\varphi_i^{-1}(A\cap \overline{O}_i)))
$$
is subanalytic in $\tilde{U}_i$ as a finite union of subanalytic sets.
Thus $f(A\cap \overline{O}_i)$ is a subanalytic subset of $Y$, for
every $i\in I$. Since $f$ is a proper map, and the cover $\{ \overline{O}_i\}_{i\in I}$
is locally finite, it follows that the family $\{ f(A\cap \overline{O}_i)\}_{i\in I}$
is also locally finite. It now follows from Lemma \ref{locfin} that
$$
f(A)=\bigcup_{i\in I}f(A\cap\overline{O}_i)
$$
is a subanalytic subset of $Y$.
\end{proof}

\begin{cor}
\label{compo}
Let $X$, $Y$ and $Z$ be real analytic orbifolds and let
$f\colon X\to Y$ and $g\colon Y\to Z$ be subanalytic  maps.
Assume $g$ is proper. Then the composed map $g\circ f\colon
X\to Z$ is subanalytic.
\end{cor}

\begin{proof}
Let ${\rm id}$ be the identity map of $X$.  By Proposition 
\ref{subprod1}, the map 
$$
{\rm id}\times g\colon X\times Y\to X\times Z
$$
is subanalytic. Since ${\rm id}\times g$ is a proper map and the set
${\rm Gr}(f)$ is subanalytic  in $X\times Y$ by Theorem \ref{graafi}, it follows
from Theorem \ref{kuvat} that ${\rm Gr}(g\circ f)=({\rm id}\times g){\rm
Gr}(f)$ is subanalytic in $X\times Z$. Since $g\circ f$ is a continuous  
orbifold map, it follows from Theorem \ref{graafi}, that $g\circ f$ is
subanalytic.
\end{proof}

\section{An alternative definition of a subanalytic set}
\label{altern}

\noindent In this section we show that subanalytic subsets of real analytic
orbifolds could in fact be defined in the same way as the 
subanalytic subsets of real analytic manifolds.

Let $X$ be a real analytic orbifold and let $A\subset X$. 
We say that $A$ is a {\it projection of
a semianalytic set} if there exists a real analytic orbifold $Y$ and a semianalytic
subset $B$ of $X\times Y$ such that $A=p(B)$, where $p\colon X\times Y\to X$
is the projection. 

\begin{theorem}
\label{orbiss}
Let $X$ be a real analytic orbifold and let $A\subset X$.
Then $A$ is a subanalytic subset of $X$ if and only if every
point $x$ of $X$ has a neighborhood $U$ such that $A\cap U$ is a 
projection of a relatively compact semianalytic set. 
\end{theorem}

\begin{proof}
First, let us assume that $A$ is a subanalytic subset of $X$.
Let $x\in X$. Then there is an orbifold chart $(\tilde{V}, G,\varphi)$ such that
$x\in V=\varphi(\tilde{V})$ and $\varphi^{-1}(A\cap V)$ is subanalytic in
$\tilde{V}$. Let $\tilde{x}\in \tilde{V}$ be such that $x=\varphi(\tilde{x})$.
Then $\tilde{x}$ has a neighborhood $W$ in $\tilde{V}$ such that there is a real
analytic manifold $M$ and a relatively compact semianalytic subset
$B$ of $\tilde{V}\times M$ with $\tilde{p}(B)=\varphi^{-1}(A\cap V)\cap W$,
where $\tilde{p}\colon \tilde{V}\times M\to \tilde{V}$ is the projection.

The product $X\times M$ is a real analytic orbifold. Let $p\colon X\times M\to
X$ be the projection and let ${\rm id}$ be the identity map of $M$. Then
$$
\varphi\circ\tilde{p}=p\circ(\varphi\times {\rm id}).
$$
The set $(\varphi\times {\rm id})(B)$ is relatively compact and it is
semianalytic since its inverse image 
$(\varphi\times {\rm id})^{-1}((\varphi\times{\rm id})(B))$
in $\tilde{V}\times M$
equals the union
$$
\bigcup_{g\in G}(gB),
$$ 
which is semianalytic as a finite union of semianalytic sets. Since
$p((\varphi\times{\rm id})(B))=(A\cap V)\cap\varphi({W})=
A\cap\varphi({W})$, 
and since $\varphi({W})$ is an open neighborhood of $x$,
the claim follows.

Let us then assume that the condition of the theorem holds. Let $x\in X$.
Then there is a real analytic orbifold $Y$ and a relatively compact
semianalytic subset $B$ of $X\times Y$ such that $A\cap U=p(B)$, where
$p\colon X\times Y\to X$ is the projection and $U$ is some neighborhood of $x$. 
Let $(\tilde{V}, G,\varphi)$ be an
orbifold chart such that $V=\varphi(\tilde{V})$ is a relatively compact semianalytic
set and $x\in V\subset \overline{V}\subset U$. Let $V'$ be an open semianalytic
neighborhood of $x$ such that $\overline{V'}\subset V$.  
By Lemma \ref{semi3}, 
$p^{-1}(V')$ is a semianalytic subset of $X\times Y$. Thus $B\cap p^{-1}(V')$
is a relatively compact semianalytic subset of $X\times Y$, and
$A\cap V'=p(B)\cap V'=p(B\cap p^{-1}(V'))$.

Since $B\cap p^{-1}(V')$ is relatively compact, it can be covered by finitely
many sets $V\times W_j$, $j=1,\ldots, n$, where the $W_j$ are basic open sets
in $Y$. Thus there are orbifold charts $(\tilde{W}_j, H_j, \psi_j)$ of $Y$ such that
$W_j=\psi_j(\tilde{W}_j)$, for every $j=1,\ldots, n$. Let $W'_j$, $j=1,\ldots, n$, be 
open semianalytic sets with $\overline{W'_j}\subset W_j$, for every $j$, and
such that the sets $V\times W'_j$ cover $B\cap p^{-1}(V)$. By Proposition \ref{huhhuh},
$$
(\varphi\times\psi_j)^{-1}( (B\cap p^{-1}(V'))\cap ({\overline{V'}}\times 
{\overline{W'_j}}))
$$
is a relatively compact semianalytic subset of $\tilde{V}\times\tilde{W}_j$,
for every $j=1\ldots, n$. Let $\tilde{p}\colon \tilde{V}\times\tilde{W}_j\to
\tilde{V}$ be the projection. By using Theorem \ref{hiro}, one can show that
$$
B_j=\varphi^{-1}(p((B\cap p^{-1}(V'))\cap (\overline{V'}\times\overline{W'_j})))=
\tilde{p}((\varphi\times\psi_j)^{-1}( (B\cap p^{-1}(V'))\cap ({\overline{V'}}\times 
{\overline{W'_j}})))
$$
is a subanalytic subset of $\tilde{V}$, for every $j$. 
It follows that
$$
\varphi^{-1}(A\cap V')=\varphi^{-1}(
p(B)\cap V')=\varphi^{-1}(
p((B\cap p^{-1}(V'))\cap \bigcup_{j=1}^n({\overline{V'}}\times 
{\overline{W'_j}})))=\bigcup_{j=1}^n B_j
$$
is a subanalytic subset of $\tilde{V}$. Consequently, $A\cap V'$ is a
subanalytic subset of $X$. By Lemma \ref{neighb}, $A$ is subanalytic.

\end{proof}

\section{The uniformization theorem}
 \label{uniformi} 
 
 \noindent In this section we prove a uniformization theorem for closed
 subanalytic subsets of real analytic quotient orbifolds. See Theorem
 0.1 in \cite{BM}, for the uniformization theorem in the manifold case.
 Our result follows from Theorem 2 in \cite{Ka2} which
 is an equivariant version of Theorem 0.1 in  \cite{BM}. 
 
 \begin{definition}
 \label{dim}
 Let $X$ be a real analytic orbifold (resp. manifold)
 and let $A$ be a subanalytic subset of $X$.
 Let $x\in A$. Then $x$ is a {\it smooth point} of $A$ of dimension $k$ if  $A\cap U$
 is a real analytic suborbifold (resp. submanifold) of dimension $k$ of $X$ 
 for some neighborhood $U$ of
 $x$. The {\it dimension} of $A$ is the highest dimension of its smooth points.
 \end{definition}
 
 \begin{lemma}
 \label{ennenunif}
 Let $G$ be a Lie group and let $M$ be a proper real analytic $G$-manifold.
 Assume the action of $G$ on $M$ is almost free. Let $\pi\colon M\to
 M/G$ be the natural projection and let $A$ be a subanalytic subset of
 $M/G$. Then $\pi^{-1}(A)$ is a subanalytic
 $G$-invariant subset of $M$ and  $\dim{\pi^{-1}(A)}=\dim{A}+\dim{G}$.
 \end{lemma}
 
 \begin{proof}  Since $\pi$ is a real analytic map, 
 it follows from Theorem \ref{kuvat}, that $\pi^{-1}(A)$ is
 subanalytic. Clearly, $\pi^{-1}(A)$ is $G$-invariant.
 Let $U$ be an open subset of $M$ such that $\pi^{-1}(A)\cap U$ is
 a real analytic manifold. Then, for every $g\in G$, $\pi^{-1}(A)\cap
 gU=g(\pi^{-1}(A)\cap U)$ is a real analytic manifold. Consequently, $\pi^{-1}(A)\cap GU$
 is a real analytic manifold and $A\cap \pi(U)$ is a real
 analytic orbifold.
 Since $G$ acts almost freely on $M$, $\dim{G/G_x}=\dim{G}$, for
 every $x\in M$.
 It follows that
 $$
 \dim{(\pi^{-1}(A)\cap GU)}=\dim{(A\cap \pi(U))}+\dim{G}. 
 $$
 Thus $\dim{\pi^{-1}(A)}\leq\dim{A}+\dim{G}$.
 
 Let then $V$ be an open subset of $M/G$ such that $A\cap V$ is a
 real analytic suborbifold of $M/G$. By Theorem \ref{apusub}, 
 $\pi^{-1}(A)\cap \pi^{-1}(V)=\pi^{-1}(A\cap V)$ is a proper real analytic 
 $G$-manifold on which $G$ acts almost freely. Thus
 $$
  \dim{(\pi^{-1}(A)\cap \pi^{-1}(V))}=\dim{(A\cap V)}+\dim{G}. 
 $$
It follows that $\dim{A}\leq \dim{\pi^{-1}(A)}-\dim{G}$, which
completes the proof.
 \end{proof}

 \begin{theorem}
 \label{important}
 Let $X$ be a real analytic quotient  orbifold. 
 Let $A$ be a closed subanalytic
 subset of $X$. Then there is a real analytic orbifold $Y$ 
 of the same dimension as $A$ and a proper
 real analytic map $f\colon Y\to X$ such that $f(Y)=A$.
 \end{theorem}
 
 \begin{proof}
 Since $X$ is a quotient orbifold, there is a real
 analytic manifold $M$ and a Lie group $G$ acting on
 $M$ real analytically, properly and with finite isotropy groups such that
 $X$ equals the orbit space $M/G$. Thus we can consider $A$
 as a subset of $M/G$. Let $\pi\colon M\to M/G$ be the natural
 projection. Then  $\pi^{-1}(A)$ is a closed subanalytic  $G$-invariant
 subset of $M$.
 By Theorem 2 in \cite{Ka2}, there exists a proper real analytic $G$-manifold
 $N$, of the same dimension as $\pi^{-1}(A)$ and a  proper real
 analytic $G$-equivariant map $f\colon N\to M$ such that $f(N)=
 \pi^{-1}(A)$.
 
 Since the isotropy group of each point of $M$ is finite and
 since $G_x\subset G_{f(x)}$, for every $x\in N$, it follows that
 the isotropy group of any point of $N$ is finite. Thus the
 orbit space $N/G$ is a real analytic quotient orbifold. The induced map
 $\bar{f}\colon N/G\to M/G=X$ is a real analytic orbifold map
 and $\bar{f}(N/G)=\pi(f(N))=A$. By Lemma \ref{ennenunif},
 $$
 \dim{N/G}=\dim{N}-\dim{G}=\dim{\pi^{-1}(A)}-\dim{G}=\dim{A},
 $$
and  the claim follows.
 \end{proof}

\end {document}